\documentclass[a4paper,10pt]{article}

\usepackage{amsfonts}
\usepackage{amsthm}
\usepackage{amsmath}
\usepackage{amssymb}
\usepackage{mathtools} 

\usepackage[top=1in, bottom=1in, left=1in, right=1in]{geometry}
\usepackage{xcolor}

\usepackage{graphicx}
\usepackage{epstopdf}
\usepackage{subfigure}

\usepackage{authblk}

\theoremstyle{plain}
\newtheorem{theorem}{Theorem}[section]
\newtheorem{corollary}[theorem]{Corollary}
\newtheorem{lemma}[theorem]{Lemma}
\newtheorem{proposition}[theorem]{Proposition}

\theoremstyle{definition}
\newtheorem{definition}[theorem]{Definition}
\newtheorem{example}[theorem]{Example}

\theoremstyle{remark}
\newtheorem{remark}[theorem]{Remark}

\DeclareSymbolFont{pxfontssymbolsC}{U}{pxsyc}{m}{n}
\DeclareMathSymbol{\coloneqq}{\mathrel}{pxfontssymbolsC}{66}

\usepackage{xcolor}
\usepackage{soul}
\definecolor{afcol}{rgb}{1,0,0}

\begin{document}


\title{Series representations for fractional-calculus operators involving generalised Mittag-Leffler functions}

\date{}

\author[1]{Arran Fernandez\thanks{Corresponding author. Email: \texttt{af454@cam.ac.uk}}}
\author[2,3]{Dumitru Baleanu\thanks{Email: \texttt{dumitru@cankaya.edu.tr}}}
\author[4,5]{H. M. Srivastava\thanks{Email: \texttt{harimsri@math.uvic.ca}}}

\affil[1]{{\small Department of Applied Mathematics and Theoretical Physics, University of Cambridge, Wilberforce Road, CB3 0WA, United Kingdom}}
\affil[2]{{\small Department of Mathematics, Cankaya University, 06530 Balgat, Ankara, Turkey}}
\affil[3]{{\small Institute of Space Sciences, Magurele-Bucharest, Romania}}
\affil[4]{{\small Department of Mathematics and Statistics, University of Victoria, Victoria, British Columbia V8W 3R4, Canada}}
\affil[5]{{\small Department of Medical Research, China Medical University Hospital, China Medical University, Taichung 40402, Taiwan, Republic of China}}

\maketitle

\begin{abstract}
We consider an integral transform introduced by Prabhakar, involving generalised multi-parameter Mittag-Leffler functions, which can be used to introduce and investigate several different models of fractional calculus. We derive a new series expression for this transform, in terms of classical Riemann--Liouville fractional integrals, and use it to obtain or verify series formulae in various specific cases corresponding to different fractional-calculus models. We demonstrate the power of our result by applying the series formula to derive analogues of the product and chain rules in more general fractional contexts. We also discuss how the Prabhakar model can be used to explore the idea of fractional iteration in connection with semigroup properties.
\end{abstract}

\section{Introduction}

Fractional calculus -- the generalisation of standard calculus to non-integer orders of differentiation and integration -- has, in its various forms, been an object of speculation and application for hundreds of years \cite{samko-kilbas-marichev,miller-ross,oldham-spanier,podlubny,kilbas-srivastava-trujillo}. Many different models of fractional calculus have been developed for different purposes, using different formulae for fractional differentiation and integration according to which one is most suited for a particular situation \cite{hilfer,atangana,west}.

The classical and most commonly-used model is the \textbf{Riemann--Liouville} one, in which fractional integrals are defined by
\begin{equation}
\label{RLdef:int}
\prescript{RL}{}I_{c+}^{\alpha}f(x)\coloneqq\frac{1}{\Gamma(\alpha)}\int_c^x(x-t)^{\alpha-1}f(t)\,\mathrm{d}t,\quad\mathrm{Re}(\alpha)>0,
\end{equation}
and fractional derivatives by
\begin{equation}
\label{RLdef:deriv}
\prescript{RL}{}D_{c+}^{\alpha}f(x)\coloneqq\frac{\mathrm{d}^n}{\mathrm{d}x^n}\Big(\prescript{RL}{}I_{c+}^{n-\alpha}f(x)\Big),\quad\mathrm{Re}(\alpha)\geq0,n\coloneqq\lfloor\mathrm{Re}(\alpha)\rfloor+1.
\end{equation}
In both cases, $\alpha\in\mathbb{C}$ is the order of differintegration (the term ``differintegration" is commonly used in fractional calculus to denote the general operation which could be either differentiation or integration according to sign) and $c$ is an arbitrary constant, the fractional analogue of a constant of integration, which is usually taken to be either $0$ or $-\infty$. The Riemann--Liouville model makes sense to use because it combines naturally together with Fourier transforms and \eqref{RLdef:int} is the natural extension of the Cauchy formula for repeated integrals. It has found diverse applications including, for example, in control theory \cite{petras}, viscoelasticity \cite{bagley}, and bioengineering \cite{magin}.

A slight modification of Riemann--Liouville is the \textbf{Caputo} model due to \cite{caputo}, in which fractional integrals are again defined by \eqref{RLdef:int} but fractional derivatives are now defined by
\begin{equation}
\label{CAPdef:deriv}
\prescript{C}{}D_{c+}^{\alpha}f(x)\coloneqq\prescript{RL}{}I_{c+}^{n-\alpha}\Big(\frac{\mathrm{d}^n}{\mathrm{d}x^n}f(x)\Big),\quad\mathrm{Re}(\alpha)\geq0,n\coloneqq\lfloor\mathrm{Re}(\alpha)\rfloor+1.
\end{equation}
This model has advantages over the Riemann--Liouville one in many applications involving initial value problems, because of the differentiation being inside the integral instead of outside; this also means that the derivative of a constant is zero. It too has found applications in many fields of science \cite{mainardi,podlubny}.

The difference between the RL and Caputo models arises from the fact that integral and differential operators do not commute. This fact can be formalised by the following theorem \cite{miller-ross,samko-kilbas-marichev}.

\begin{lemma}
\label{RLcomp}
If $f$ is an $L^1$ function, then for any $\alpha,\beta\in\mathbb{C}$ with $\mathrm{Re}(\beta)>0$, we have
\begin{equation}
\label{RLcomp:int}
\prescript{RL}{}I_{c+}^{\alpha}\prescript{RL}{}I_{c+}^{\beta}f=\prescript{RL}{}I_{c+}^{\alpha+\beta}f.
\end{equation}
If in addition $f$ is a $C^n$ function for some $n\in\mathbb{N}$, then we have
\begin{equation}
\label{RLcomp:deriv}
\frac{\mathrm{d}^n}{\mathrm{d}x^n}\left(\prescript{RL}{}D_{c+}^{\alpha}f(x)\right)=\prescript{RL}{}D_{c+}^{\alpha+n}f(x)=\prescript{RL}{}D_{c+}^{\alpha}f^{(n)}(x)+\sum_{k=1}^n\frac{(x-c)^{-\alpha-k}}{\Gamma(1-\alpha-k)}f^{(n-k)}(c).
\end{equation}
\end{lemma}

The terms \textit{differintegral} and \textit{differintegration} are commonly used in fractional calculus to denote the generalised operation which could be either differentiation or integration according to the sign of the real part of the order.

An operator introduced by Prabhakar in 1971 \cite{prabhakar} for solving a particular singular integral equation has also been adapted as a fractional differintegral operator \cite{kilbas-saigo-saxena}, and its properties and applications have been investigated in papers such as \cite{kilbas-saigo-saxena,garra-garrappa,giusti-colombaro}. The \textbf{Prabhakar} fractional integral is defined by
\begin{equation}
\label{PRABdef:int}
\mathcal{E}_{\alpha,\beta;c+}^{\omega,\rho}f(x)\coloneqq\int_c^x(x-t)^{\beta-1}E_{\alpha,\beta}^{\rho}\left[\omega(x-t)^{\alpha}\right]f(t)\,\mathrm{d}t,\quad\mathrm{Re}(\alpha)>0,\mathrm{Re}(\beta)>0,
\end{equation}
where the generalised Mittag-Leffler function $E_{\alpha,\beta}^{\rho}$ is defined by
\begin{equation}
\label{GMLdef}
E_{\alpha,\beta}^{\rho}(z)\coloneqq\sum_{n=0}^{\infty}\frac{(\rho)_nz^n}{\Gamma(\alpha n+\beta)n!}=\sum_{n=0}^{\infty}\frac{\Gamma(\rho+n)z^n}{\Gamma(\rho)\Gamma(\alpha n+\beta)n!}.
\end{equation}
The Prabhakar integral is known \cite{kilbas-saigo-saxena} to be a bounded operator on $L^1$ functions. It has also been shown \cite{kilbas-saigo-saxena} that its left inverse can be given by any expression of the following form:
\begin{equation}
\label{PRAB:inverse}
\left[\mathcal{E}_{\alpha,\beta;c+}^{\omega,\rho}\right]^{-1}f(x)=\prescript{RL}{}D_{c+}^{\beta+\gamma}\mathcal{E}_{\alpha,\gamma;c+}^{\omega,-\rho}f(x),\quad\gamma\in\mathbb{C},\mathrm{Re}(\gamma)>0.
\end{equation}
In particular, and by analogy with the definition \eqref{RLdef:deriv} of Riemann--Liouville fractional derivatives, we can define the Prabhakar fractional differential operator as follows:
\begin{equation}
\label{PRABdef:deriv}
\mathcal{D}_{\alpha,\beta;c+}^{\omega,\rho}f(x)\coloneqq\frac{\mathrm{d}^m}{\mathrm{d}x^m}\mathcal{E}_{\alpha,m-\beta;c+}^{\omega,-\rho}f(x),\quad\mathrm{Re}(\alpha)>0,\mathrm{Re}(\beta)>0,m\coloneqq\lfloor\mathrm{Re}(\beta)\rfloor+1.
\end{equation}
It is also possible \cite{garra-gorenflo-polito-tomovski} to define a Prabhakar fractional derivative of Caputo type, in a way analogous to \eqref{CAPdef:deriv}. But here we only need to consider the Prabhakar fractional derivative of Riemann--Liouville type, as defined by \eqref{PRABdef:deriv}.

The Prabhakar operator has also been extended and generalised still further \cite{srivastava-tomovski,garra-gorenflo-polito-tomovski}. The \textbf{generalised Prabhakar} fractional integral is defined by
\begin{multline}
\label{GPdef:int}
\mathcal{E}_{\alpha,\beta;c+}^{\omega,\rho,\kappa}f(x)\coloneqq\int_c^x(x-t)^{\beta-1}E_{\alpha,\beta}^{\rho,\kappa}\left[\omega(x-t)^{\alpha}\right]f(t)\,\mathrm{d}t,\\ \min\left(\mathrm{Re}(\alpha),\mathrm{Re}(\beta),\mathrm{Re}(\kappa)\right)>0,\mathrm{Re}(\kappa-\alpha)<1,
\end{multline}
where the generalised Mittag-Leffler function $E_{\alpha,\beta}^{\rho,\kappa}$ is defined by
\begin{equation}
\label{GGMLdef}
E_{\alpha,\beta}^{\rho,\kappa}(z)\coloneqq\sum_{n=0}^{\infty}\frac{(\rho)_{\kappa n}z^n}{\Gamma(\alpha n+\beta)n!}=\sum_{n=0}^{\infty}\frac{\Gamma(\rho+\kappa n)z^n}{\Gamma(\rho)\Gamma(\alpha n+\beta)n!}.
\end{equation}
The generalised Prabhakar integral is known \cite{srivastava-tomovski} to be a bounded operator on $L^1$ functions.

We recall that the Mittag-Leffler function is known to be very significant in fractional calculus \cite{mainardi-gorenflo,mathai-haubold,tomovski-hilfer-srivastava,srivastava}, and its properties have been exhaustively studied in this connection \cite{gorenflo-kilbas-mainardi-rogosin,haubold-mathai-saxena}. Prabhakar operators in particular have already found applications in diffusion, relaxation, and stochastic processes \cite{sandev,garrappa,garrappa-mainardi-maione,polito-tomovski}.

More recently still, other models of fractional calculus have been proposed which are less general than the Prabhakar model but easier to compute and with more direct applications.

In the \textbf{Atangana--Baleanu} model, introduced in \cite{atangana-baleanu}, fractional derivatives are defined either in a Riemann--Liouville sense by
\begin{equation}
\label{ABRdef:deriv}
\prescript{ABR}{}D_{c+}^{\alpha}f(x)\coloneqq\frac{B(\alpha)}{1-\alpha}\frac{\mathrm{d}}{\mathrm{d}x}\int_c^xE_{\alpha}\left(\frac{-\alpha}{1-\alpha}(x-t)^{\alpha}\right)f(t)\,\mathrm{d}t,\quad0<\alpha<1,
\end{equation}
or in a Caputo sense by
\begin{equation}
\label{ABCdef:deriv}
\prescript{ABC}{}D_{c+}^{\alpha}f(x)\coloneqq\frac{B(\alpha)}{1-\alpha}\int_c^xE_{\alpha}\left(\frac{-\alpha}{1-\alpha}(x-t)^{\alpha}\right)f'(t)\,\mathrm{d}t,\quad0<\alpha<1,
\end{equation}
while fractional integrals are defined by
\begin{equation}
\label{ABdef:int}
\prescript{AB}{}I_{c+}^{\alpha}f(x)\coloneqq\frac{1-\alpha}{B(\alpha)}f(x)+\frac{\alpha}{B(\alpha)}\prescript{RL}{}I_{c+}^{\alpha}f(x),\quad0<\alpha<1.
\end{equation}
In each case, $B(\alpha)$ is a multiplier function satisfying $B(0)=B(1)=1$ and assumed \cite{baleanu-fernandez} to be real and positive, while $E_{\alpha}$ is the standard Mittag-Leffler function defined by
\begin{equation}
\label{MLdef}
E_{\alpha}(z)\coloneqq\sum_{n=0}^{\infty}\frac{z^n}{\Gamma(\alpha n+1)}=E_{\alpha,1}^{1,1}(z).
\end{equation}
The AB model was motivated by similar considerations as the Caputo--Fabrizio one \cite{caputo-fabrizio}, with special consideration of modelling the dynamics of non-local systems in ways that could not be done using the classical definitions of fractional calculus. Its applications to various fields of science have been explored in \cite{atangana-baleanu,atangana,alkahtani,srivastava-saad}, and its theoretical side has been developed in \cite{abdeljawad-baleanu,fernandez-baleanu2,djida-atangana-area}.

Starting from the AB model, the first two authors have also defined a two-parameter fractional differintegral operator \cite{fernandez-baleanu}. In the \textbf{iterated AB} model, fractional differintegrals are defined by
\begin{equation}
\label{IABdef}
\prescript{IAB}{}I_{c+}^{\alpha,\rho}f(x)\coloneqq\sum_{n=0}^{\infty}\frac{\binom{\rho}{n}(1-\alpha)^{\rho-n}\alpha^n}{B(\alpha)^{\rho}}\prescript{RL}{}I_{c+}^{\alpha n}f(x),\quad0\leq\alpha\leq1,\rho\in\mathbb{R}.
\end{equation}
This can be thought of as the $\rho$th iteration of the $\alpha$th AB fractional integral: in some sense, \eqref{IABdef} equates to \[\prescript{IAB}{}I_{c+}^{\alpha,\rho}f(x)=\left(\prescript{AB}{}I_{c+}^{\alpha}\right)^{\rho}f(t)\]

In the current work, we take inspiration from existing results on the AB model \cite{baleanu-fernandez} to prove analogous results for the Prabhakar and generalised Prabhakar models: specifically, a series formula for these fractional differintegral operators in terms of only classical (RL) fractional integrals. We demonstrate how this series formula can be used, both to give quick alternative proofs for several known results on Prabhakar differintegrals, and also to derive new results in the Prabhakar model such as analogues of the product rule and chain rule. We also expect that the series formula will make numerical computation of Prabhakar derivatives easier than before, since now all that needs to be considered is Riemann--Liouville fractional integrals, not special functions such as the generalised Mittag-Leffler function.

We also demonstrate how these new results for the Prabhakar model lead to similar results for all of the other models of fractional calculus considered above. The CF, AB, and iterated AB models can all be seen as special cases of Prabhakar, which means that any theorem proved in the Prabhakar model automatically gives rise to new results in these other models, which can be directly applied to various real-world problems.

Finally, inspired by \cite{fernandez-baleanu} and the definition of the iterated AB model, we consider what happens when we try to iterate the Prabhakar fractional differintegral, and how to fit this idea into the existing framework of the Prabhakar model.

\section{Series expressions}

In this section, we prove a series formula for the generalised Prabhakar integral defined by \eqref{GPdef:int}, and use this to deduce analogous formulae for the more specific Prabhakar and AB differintegrals, as well as to re-derive various fundamental results about the generalised Prabhakar operators.

\begin{theorem}
\label{GPseries}
Under the conditions \eqref{GGML:conditions} on the parameters $\alpha,\beta,\omega,\rho,\kappa$, for any interval $(c,d)\subset\mathbb{R}$ and any function $f\in L^1(c,d)$, the generalised Prabhakar operator can be written as
\begin{equation}
\label{GPseries:eqn}
\mathcal{E}_{\alpha,\beta;c+}^{\omega,\rho,\kappa}f(x)=\sum_{n=0}^{\infty}\frac{\Gamma(\rho+\kappa n)\omega^n}{\Gamma(\rho)n!}\prescript{RL}{}I_{c+}^{\alpha n+\beta}f(x),
\end{equation}
where the series on the right-hand side is locally uniformly convergent.
\end{theorem}

\begin{proof}
Our starting point is the series formula \eqref{GGMLdef} for the generalised Mittag-Leffler function. This series is known \cite{kilbas-saigo-saxena,lavault} to be locally uniformly convergent in $z$, provided that
\begin{equation}
\label{GGML:conditions}
\mathrm{Re}(\alpha)>0,\quad\mathrm{Re}(\beta)>0,\quad\mathrm{Re}(\kappa)>0,\quad\mathrm{Re}(\kappa-\alpha)<1.
\end{equation}
Thus we can interchange the summation and integration in the formula \eqref{GPdef:int}, and work as follows:
\begin{align*}
\mathcal{E}_{\alpha,\beta;c+}^{\omega,\rho,\kappa}f(x)&=\int_c^x(x-t)^{\beta-1}E_{\alpha,\beta}^{\rho,\kappa}\left[\omega(x-t)^{\alpha}\right]f(t)\,\mathrm{d}t \\
&=\int_c^x(x-t)^{\beta-1}\sum_{n=0}^{\infty}\frac{(\rho)_{\kappa n}\omega^n(x-t)^{\alpha n}}{\Gamma(\alpha n+\beta)n!}f(t)\,\mathrm{d}t \\
&=\sum_{n=0}^{\infty}\int_c^x\frac{\Gamma(\rho+\kappa n)\omega^n(x-t)^{\alpha n+\beta-1}}{\Gamma(\rho)\Gamma(\alpha n+\beta)n!}f(t)\,\mathrm{d}t \\
&=\sum_{n=0}^{\infty}\frac{\Gamma(\rho+\kappa n)\omega^n}{\Gamma(\rho)n!}\cdot\frac{1}{\Gamma(\alpha n+\beta)}\int_c^x(x-t)^{\alpha n+\beta-1}f(t)\,\mathrm{d}t \\
&=\sum_{n=0}^{\infty}\frac{\Gamma(\rho+\kappa n)\omega^n}{\Gamma(\rho)n!}\prescript{RL}{}I_{c+}^{\alpha n+\beta}f(x).
\end{align*}
Now we have expressed the generalised Prabhakar operator as a series of Riemann--Liouville fractional integrals, and the result is established.
\end{proof}

\begin{remark}
An almost identical result to Theorem \ref{GPseries}, namely Corollary \ref{PRABseries} below, as well as the further Corollary \ref{ABseries}, was discovered independently in \cite{giusti}, also in 2018. Both our approach here and that used in \cite{giusti}, for Mittag-Leffler functions of 4 and 3 parameters respectively, are very similar to the approach previously used in \cite{baleanu-fernandez} for Mittag-Leffler functions of 1 parameter. We note that a recent addendum to \cite{giusti}, namely \cite{giusti2}, does cite the previous paper \cite{baleanu-fernandez}.
\end{remark}

In order to deduce similar results in other models of fractional calculus, we first note the following equivalences.

\begin{proposition}
\label{GPcases}
The Prabhakar, AB, and iterated AB models of fractional calculus can all be viewed as special cases of the generalised Prabhakar model \eqref{GPdef:int}, as follows.
\begin{align}
\label{GP:PRAB} \mathcal{E}_{\alpha,\beta;c+}^{\omega,\rho}f(x)&=\mathcal{E}_{\alpha,\beta;c+}^{\omega,\rho,1}f(x),&&\quad\mathrm{Re}(\alpha)>0,\mathrm{Re}(\beta)>0; \\
\label{GP:ABR} \prescript{ABR}{}D_{c+}^{\alpha}f(x)&=\frac{B(\alpha)}{1-\alpha}\cdot\frac{\mathrm{d}}{\mathrm{d}x}\mathcal{E}_{\alpha,1;c+}^{\frac{-\alpha}{1-\alpha},1,1}f(x),&&\quad0<\alpha<1; \\
\label{GP:ABC} \prescript{ABC}{}D_{c+}^{\alpha}f(x)&=\frac{B(\alpha)}{1-\alpha}\mathcal{E}_{\alpha,1;c+}^{\frac{-\alpha}{1-\alpha},1,1}f'(x),&&\quad0<\alpha<1; \\
\label{GP:IAB} \prescript{IAB}{}I_{c+}^{\alpha,\rho}f(x)&=\left(\frac{1-\alpha}{B(\alpha)}\right)^{\rho}\mathcal{E}_{\alpha,0;c+}^{\frac{\alpha}{1-\alpha},\rho+1,1}f(x),&&\quad0\leq\alpha\leq1,\rho\in\mathbb{R}.
\end{align}
\end{proposition}

\begin{proof}
This follows directly from comparing the definitions \eqref{PRABdef:int}, 
\eqref{ABRdef:deriv}, \eqref{ABCdef:deriv}, \eqref{IABdef} with the formula \eqref{GPdef:int} for the generalised Prabhakar integral.
\end{proof}

Using Proposition \ref{GPcases}, it is straightforward to deduce the following corollaries of Theorem \ref{GPseries} for the other models of fractional calculus which can be seen as special cases of generalised Prabhakar.

\begin{corollary}
\label{PRABseries}
Given complex parameters $\alpha,\beta,\omega,\rho$ satisfying $\mathrm{Re}(\alpha)>0,\mathrm{Re}(\beta)>0$, then for any interval $(c,d)\subset\mathbb{R}$ and any function $f\in L^1(c,d)$, the Prabhakar fractional integral can be written as
\begin{equation}
\label{PRABseries:int}
\mathcal{E}_{\alpha,\beta;c+}^{\omega,\rho}f(x)=\sum_{n=0}^{\infty}\frac{\Gamma(\rho+n)\omega^n}{\Gamma(\rho)n!}\prescript{RL}{}I_{c+}^{\alpha n+\beta}f(x),
\end{equation}
and the Prabhakar fractional derivative can be written as
\begin{equation}
\label{PRABseries:deriv}
\mathcal{D}_{\alpha,\beta;c+}^{\omega,\rho}f(x)=\sum_{n=0}^{\infty}\frac{\Gamma(-\rho+n)\omega^n}{\Gamma(-\rho)n!}\prescript{RL}{}I_{c+}^{\alpha n-\beta}f(x),
\end{equation}
where the series on the right-hand sides are locally uniformly convergent.
\end{corollary}

We note that if \eqref{PRABseries:int}-\eqref{PRABseries:deriv} are used as the definitions of Prabhakar fractional integrals and derivatives, then differintegrals in this model can be unified under a single series formula, where switching between derivatives to integrals means simply switching the sign of the parameters $\rho$ and $\beta$.

\begin{corollary}
\label{ABseries}
Given $\alpha\in(0,1)$, then for any interval $(c,d)\subset\mathbb{R}$ and any function $f\in L^1(c,d)$, the AB fractional derivatives of Riemann--Liouville and Caputo type can be written respectively as:
\begin{align}
\label{ABRseries:eqn}
\prescript{ABR}{}D_{c+}^{\alpha}f(x)&=\frac{B(\alpha)}{1-\alpha}\cdot\frac{\mathrm{d}}{\mathrm{d}x}\sum_{n=0}^{\infty}\left(\frac{-\alpha}{1-\alpha}\right)^n\prescript{RL}{}I_{c+}^{\alpha n+1}f(x), \\
\label{ABCseries:eqn}
\prescript{ABC}{}D_{c+}^{\alpha}f(x)&=\frac{B(\alpha)}{1-\alpha}\sum_{n=0}^{\infty}\left(\frac{-\alpha}{1-\alpha}\right)^n\prescript{RL}{}I_{c+}^{\alpha n+1}f'(x),
\end{align}
where the series on the right-hand sides are locally uniformly convergent.
\end{corollary}

\begin{remark}
We note that Corollary \ref{ABseries} was already proved in \cite{baleanu-fernandez}, and was in fact what inspired this work; we mention it here only to demonstrate that the new result is more general.

For the iterated AB differintegral, Theorem \ref{GPseries} gives us no new result, since substituting \eqref{GP:IAB} into \eqref{GPseries:eqn} would yield the same series which was used in \cite{fernandez-baleanu} to define the operator in the first place. However, the idea of iterating fractional differintegral models to fractional powers can still be used in the Prabhakar model; see section 5 below.
\end{remark}

Now that we have established all these series formulae for differintegrals in assorted fractional calculus models, it is time to ask what such formulae can be used for.

First of all, they make much easier the proofs of many fundamental results about Prabhakar operators and their relationships with classical fractional-calculus operators. For example, the following result comprises Theorems 3 and 4 in \cite{prabhakar}, Theorems 6 and 7 in \cite{kilbas-saigo-saxena}, and Theorems 4 and 5 in \cite{srivastava-tomovski}. Each of these identities was proved originally using Fubini's theorem, but now follows much more quickly from our new series formula.

\begin{theorem}
\label{GPwithRL}
The generalised Prabhakar operator \eqref{GPdef:int}, with $\alpha,\beta,\omega,\rho,\kappa\in\mathbb{C}$ satisfying \eqref{GGML:conditions}, interacts naturally with Riemann--Liouville differintegral operators in the following ways.

For any $L^1$ function $f$, and any $\mu\in\mathbb{C}$ such that $\mathrm{Re}(\mu)>-\mathrm{Re}(\beta)$, we have:
\begin{equation}
\label{GPwithRL:eqn1}
\prescript{RL}{}I^{\mu}_{c+}\mathcal{E}_{\alpha,\beta;c+}^{\omega,\rho,\kappa}f=\mathcal{E}_{\alpha,\beta+\mu;c+}^{\omega,\rho,\kappa}f.
\end{equation}
If in addition $\mathrm{Re}(\mu)>0$, then
\begin{equation}
\label{GPwithRL:eqn2}
\prescript{RL}{}I^{\mu}_{c+}\mathcal{E}_{\alpha,\beta;c+}^{\omega,\rho,\kappa}f=\mathcal{E}_{\alpha,\beta;c+}^{\omega,\rho,\kappa}\prescript{RL}{}I^{\mu}_{c+}f.
\end{equation}
\end{theorem}

\begin{proof}
For the first identity, by the series formula \eqref{GPseries:eqn} it is enough to show that \[\prescript{RL}{}I^{\mu}_{c+}\prescript{RL}{}I^{\alpha n+\beta}_{c+}f=\prescript{RL}{}I^{\alpha n+\beta+\mu}_{c+}f,\quad n\geq0,\] which is clearly true by basic properties of Riemann--Liouville differintegrals. For the second identity, by \eqref{GPseries:eqn} it is enough to show that \[\prescript{RL}{}I^{\mu}_{c+}\prescript{RL}{}I^{\alpha n+\beta}_{c+}f=\prescript{RL}{}I^{\alpha n+\beta}_{c+}\prescript{RL}{}I^{\mu}_{c+}f,\quad n\geq0,\] which again is clearly true, by the semigroup property for Riemann--Liouville fractional integrals.
\end{proof}

\begin{remark}
It is important to note that \eqref{GPwithRL:eqn2} is \textit{not} always valid when $\mathrm{Re}(\mu)<0$. It will be valid under certain initial value conditions on $f$, but in general the left and right hand sides differ by a series of initial value terms, just as in the Riemann--Liouville case -- see Lemma \ref{RLcomp}.
\end{remark}

The following result comprises Theorem 5 in \cite{prabhakar} and Theorem 8 in \cite{kilbas-saigo-saxena}, both of which were originally proved using Fubini's theorem. It can now be proved in a more elementary way using our Theorem \ref{GPseries}.

\begin{theorem}
\label{PRABsemigp}
The Prabhakar operator \eqref{PRABdef:int} satisfies the following semigroup property, under the usual assumptions $\mathrm{Re}(\alpha)>0,\mathrm{Re}(\beta_1)>0,\mathrm{Re}(\beta_2)>0$ and operating on an $L^1$ function $f$.
\begin{equation}
\label{PRABsemigp:eqn}
\mathcal{E}_{\alpha,\beta_1;c+}^{\omega,\rho_1}\mathcal{E}_{\alpha,\beta_2;c+}^{\omega,\rho_2}f=\mathcal{E}_{\alpha,\beta_1+\beta_2;c+}^{\omega,\rho_1+\rho_2}f.
\end{equation}
\end{theorem}

\begin{proof}
Using the series formula \eqref{PRABseries:int}, the left-hand side of \eqref{PRABsemigp:eqn} becomes
\begin{align*}
\mathcal{E}_{\alpha,\beta_1;c+}^{\omega,\rho_1}\mathcal{E}_{\alpha,\beta_2;c+}^{\omega,\rho_2}f(x)&=\sum_{n=0}^{\infty}\frac{\Gamma(\rho_1+n)\omega^n}{\Gamma(\rho_1)n!}\prescript{RL}{}I_{c+}^{\alpha n+\beta_1}\left(\sum_{m=0}^{\infty}\frac{\Gamma(\rho_2+m)\omega^m}{\Gamma(\rho_2)m!}\prescript{RL}{}I_{c+}^{\alpha m+\beta_2}f(x)\right) \\
&=\sum_{m,n}\frac{\Gamma(\rho_1+n)\Gamma(\rho_2+m)\omega^{m+n}}{\Gamma(\rho_1)\Gamma(\rho_2)n!m!}\prescript{RL}{}I_{c+}^{\alpha(m+n)+\beta_1+\beta_2}f(x) \\
&=\sum_{k=0}^{\infty}\left[\sum_{m+n=k}\frac{B(\rho_1+n,\rho_2+m)(m+n)!}{B(\rho_1,\rho_2)n!m!}\right]\frac{\Gamma(\rho_1+\rho_2+k)\omega^k}{\Gamma(\rho_1+\rho_2)k!}\prescript{RL}{}I_{c+}^{\alpha k+\beta_1+\beta_2}f(x),
\end{align*}
where $B$ is the beta function. Similarly, the right-hand side becomes
\begin{equation*}
\mathcal{E}_{\alpha,\beta_1+\beta_2;c+}^{\omega,\rho_1+\rho_2}f(x)=\sum_{k=0}^{\infty}\frac{\Gamma(\rho_1+\rho_2+k)\omega^k}{\Gamma(\rho_1+\rho_2)k!}\prescript{RL}{}I_{c+}^{\alpha k+\beta_1+\beta_2}f(x).
\end{equation*}
So it suffices to prove that
\begin{equation*}
\sum_{m+n=k}\frac{B(\rho_1+n,\rho_2+m)(m+n)!}{B(\rho_1,\rho_2)n!m!}=1,
\end{equation*}
which can be verified by induction on $k$, using the fact that \[B(x,y)=\frac{x-1}{x+y-1}B(x-1,y)+\frac{y-1}{x+y-1}B(x,y-1).\]
\end{proof}

The following result comprises Theorem 9 in \cite{kilbas-saigo-saxena} and our equation \eqref{PRAB:inverse}, and it is also stressed in \cite{garra-gorenflo-polito-tomovski}. It is used to justify the definition \eqref{PRABdef:deriv} of Prabhakar fractional derivatives.

\begin{theorem}
Under the usual assumptions $\mathrm{Re}(\alpha)>0,\mathrm{Re}(\beta)>0$, the Prabhakar fractional integral defined by \eqref{PRABdef:int} has a left inverse on the space of $L^1$ functions. This left inverse can be defined by the expression \eqref{PRAB:inverse}, which is independent of the value of $\gamma$ and therefore equivalent to the Prabhakar fractional derivative \eqref{PRABdef:deriv}.
\end{theorem}

\begin{proof}
Using first the result of Theorem \ref{PRABsemigp} and then the series formula \eqref{PRABseries:int}, we find that
\begin{align*}
\prescript{RL}{}D_{c+}^{\beta+\gamma}\mathcal{E}_{\alpha,\gamma;c+}^{\omega,-\rho}\mathcal{E}_{\alpha,\beta;c+}^{\omega,\rho}f(x)&=\prescript{RL}{}D_{c+}^{\beta+\gamma}\mathcal{E}_{\alpha,\beta+\gamma;c+}^{\omega,0}f(x) \\
&=\prescript{RL}{}D_{c+}^{\beta+\gamma}\left(\sum_{n=0}^{\infty}\frac{\Gamma(0+n)\omega^n}{\Gamma(0)n!}\prescript{RL}{}I_{c+}^{\alpha n+\beta+\gamma}f(x)\right) \\
&=\prescript{RL}{}D_{c+}^{\beta+\gamma}\left(\prescript{RL}{}I_{c+}^{\beta+\gamma}f(x)\right) \\
&=f(x),
\end{align*}
where in the final line we used \eqref{RLcomp:int} from Lemma \ref{RLcomp}. Thus \eqref{PRAB:inverse} provides a left inverse to the Prabhakar fractional integral, for any value of $\gamma$. To check that this expression is independent of $\gamma$, we use the series formula \eqref{PRABseries:int} again to get:
\begin{align*}
\prescript{RL}{}D_{c+}^{\beta+\gamma}\mathcal{E}_{\alpha,\gamma;c+}^{\omega,-\rho}f(x)&=\prescript{RL}{}D_{c+}^{\beta+\gamma}\left(\sum_{n=0}^{\infty}\frac{\Gamma(-\rho+n)\omega^n}{\Gamma(-\rho)n!}\prescript{RL}{}I_{c+}^{\alpha n+\gamma}f(x)\right) \\
&=\sum_{n=0}^{\infty}\frac{\Gamma(-\rho+n)\omega^n}{\Gamma(-\rho)n!}\prescript{RL}{}I_{c+}^{\alpha n-\beta}f(x),
\end{align*}
which is independent of $\gamma$ and precisely equal to the series formula \eqref{PRABseries:deriv} for Prabhakar fractional derivatives.
\end{proof}

\section{Fractional extensions of the product rule}

The product rule, or Leibniz rule, is one of the most fundamental results about classical derivatives, and there exist many papers on generalising it to various fractional scenarios. In particular, the work of Osler \cite{osler} established the following version for Riemann--Liouville fractional differintegrals:
\begin{equation}
\label{product:RL}
\prescript{RL}{}D^{\alpha}_{c+}\big(f(x)g(x)\big)=\sum_{m=0}^{\infty}\binom{\alpha}{m}\prescript{RL}{}D^{\alpha-m}_{c+}f(x)\frac{\mathrm{d}^mg(x)}{\mathrm{d}x^m},\quad\alpha\in\mathbb{C},x\in U\backslash\{c\},
\end{equation}
where $f(x)$, $g(x)$, and $f(x)g(x)$ are all functions in the form $x^{\zeta}\xi(x)$ with $\mathrm{Re}(\zeta)>-1$ and $\xi$ holomorphic on a complex domain $U\ni c$. Using the series formula for AB fractional derivatives, it is possible \cite{baleanu-fernandez} to derive from \eqref{product:RL} a version of the fractional product rule which is valid in the AB model:
\begin{equation}
\label{product:AB}
\prescript{ABR}{}D^{\alpha}_{c+}\big(f(x)g(x)\big)=\sum_{m=0}^{\infty}\frac{\mathrm{d}^mg(x)}{\mathrm{d}x^m}\Bigg[\frac{B(\alpha)}{1-\alpha}\sum_{n=0}^{\infty}\Big(\frac{-\alpha}{1-\alpha}\Big)^n\binom{-n\alpha}{m}\prescript{RL}{}I^{\alpha n+m}_{c+}f(x)\Bigg],
\end{equation}
where $0<\alpha<1$ and $f$ and $g$ are as before.

Now, using the new series formula for generalised Prabhakar fractional differintegrals, we can similarly derive a version of the product rule in this general model of fractional calculus.

\begin{theorem}
\label{product:GP}
Let $f$ and $g$ be complex functions such that $f(x)$, $g(x)$, and $f(x)g(x)$ are all in the form $x^{\zeta}\xi(x)$ with $\mathrm{Re}(\zeta)>-1$ and $\xi$ holomorphic on a domain $U\subset\mathbb{C}$. Then for any complex parameters $\alpha,\beta,\omega,\rho,\kappa$ satisfying the conditions \eqref{GGML:conditions}, the generalised Prabhakar operator satisfies the following version of the product rule:
\begin{equation}
\label{product:GP:eqn}
\mathcal{E}_{\alpha,\beta;c+}^{\omega,\rho,\kappa}\big(f(x)g(x)\big)=\sum_{m=0}^{\infty}\frac{\mathrm{d}^mg(x)}{\mathrm{d}x^m}\left[\sum_{n=0}^{\infty}\frac{\Gamma(\rho+\kappa n)\Gamma(1-\beta-\alpha n)\omega^n}{\Gamma(\rho)\Gamma(1-\beta-\alpha n-m)m!n!}\prescript{RL}{}I_{c+}^{\alpha n+\beta+m}f(x)\right].
\end{equation}
\end{theorem}

\begin{proof}
Formally, we can argue as follows, using the series formula \eqref{GPseries:eqn} for generalised Prabhakar integrals and the Riemann--Liouville version \eqref{product:RL} of the product rule:
\begin{align*}
\mathcal{E}_{\alpha,\beta;c+}^{\omega,\rho,\kappa}\big(f(x)g(x)\big)&=\sum_{n=0}^{\infty}\frac{\Gamma(\rho+\kappa n)\omega^n}{\Gamma(\rho)n!}\prescript{RL}{}I_{c+}^{\alpha n+\beta}\big(f(x)g(x)\big) \\
&=\sum_{n=0}^{\infty}\frac{\Gamma(\rho+\kappa n)\omega^n}{\Gamma(\rho)n!}\left[\sum_{m=0}^{\infty}\binom{-\alpha n-\beta}{m}\prescript{RL}{}I^{\alpha n+\beta+m}_{c+}f(x)\frac{\mathrm{d}^mg(x)}{\mathrm{d}x^m}\right] \\
&=\sum_{m=0}^{\infty}\frac{\mathrm{d}^mg(x)}{\mathrm{d}x^m}\left[\sum_{n=0}^{\infty}\frac{\Gamma(\rho+\kappa n)\omega^n}{\Gamma(\rho)n!}\binom{-\alpha n-\beta}{m}\prescript{RL}{}I^{\alpha n+\beta+m}_{c+}f(x)\right],
\end{align*}
which yields the required expression. To make this proof rigorous, we just need to verify local uniform convergence of the double series found above.

The following finite truncation of \eqref{product:RL} can be found as equation (2.199) in \cite{podlubny}, and we use it as our starting point:
\begin{equation}
\label{product:remainder}
\prescript{RL}{}D^{\alpha}_{c+}\big(f(x)g(x)\big)=\sum_{m=0}^{N}\binom{\alpha}{m}\prescript{RL}{}D^{\alpha-m}_{c+}f(x)\frac{\mathrm{d}^mg(x)}{\mathrm{d}x^m}-R_N^{\alpha}(x),\quad\alpha\in\mathbb{R},N\geq\alpha+1,
\end{equation}
where we assume $f\in C[a,x],g\in C^{N+1}[a,x]$, and the remainder term $R_N^{\alpha}(x)$ is defined by \[R_N^{\alpha}(x)=\frac{1}{N!\Gamma(-\alpha)}\int_c^x(x-y)^{-\alpha-1}f(y)\bigg[\int_{y}^xg^{(N+1)}(\xi)(y-\xi)^N\,\mathrm{d}\xi\bigg]\mathrm{d}y.\] Here we replace $\alpha$ in \eqref{product:remainder} by $-\alpha n-\beta$, and substitute the resulting expression into the series formula for generalised Prabhakar integrals:
\begin{align*}
\mathcal{E}_{\alpha,\beta;c+}^{\omega,\rho,\kappa}\big(f(x)g(x)\big)&=\sum_{n=0}^{\infty}\frac{\Gamma(\rho+\kappa n)\omega^n}{\Gamma(\rho)n!}\prescript{RL}{}I_{c+}^{\alpha n+\beta}\big(f(x)g(x)\big) \\
&=\sum_{n=0}^{\infty}\frac{\Gamma(\rho+\kappa n)\omega^n}{\Gamma(\rho)n!}\left[\sum_{m=0}^{N}\binom{-\alpha n-\beta}{m}\prescript{RL}{}D^{-\alpha n-\beta-m}_{c+}f(x)\frac{\mathrm{d}^mg(x)}{\mathrm{d}x^m}-R_N^{-\alpha n-\beta}(x)\right] \\
\begin{split}
&=\sum_{m=0}^{N}\frac{\mathrm{d}^mg(x)}{\mathrm{d}x^m}\left[\sum_{n=0}^{\infty}\frac{\Gamma(\rho+\kappa n)\omega^n}{\Gamma(\rho)n!}\binom{-\alpha n-\beta}{m}\prescript{RL}{}I^{\alpha n+\beta+m}_{c+}f(x)\right] \\
&\hspace{7cm}-\sum_{n=0}^{\infty}\frac{\Gamma(\rho+\kappa n)\omega^n}{\Gamma(\rho)n!}R_N^{-\alpha n-\beta}(x),
\end{split}
\end{align*}
where swapping the sums in the last step is justified because the sum over $m$ is finite and the sum over $n$ is locally uniformly convergent by Theorem \ref{GPseries}. The only thing left to prove now is that
\begin{equation}
\label{product:limit}
\lim_{N\rightarrow\infty}\sum_{n=0}^{\infty}\frac{\Gamma(\rho+\kappa n)\omega^n}{\Gamma(\rho)n!}R_N^{-\alpha n-\beta}(x)=0.
\end{equation}
To show this, we use an argument similar to that used in \cite{podlubny} and \cite{baleanu-fernandez}. Specifically, equation (2.201) from \cite{podlubny} tells us that \[R_N^{-\alpha n-\beta}(x)=\frac{(-1)^N(x-c)^{N+\alpha n+\beta+1}}{N!\Gamma(\alpha n+\beta)}\int_0^1\int_0^1f\big(c+\eta(x-c)\big)g^{(N+1)}\big(c+(\zeta+\eta-\zeta\eta)(x-c)\big)\,\mathrm{d}\eta\,\mathrm{d}\zeta,\] which means that
\begin{multline}
\sum_{n=0}^{\infty}\frac{\Gamma(\rho+\kappa n)\omega^n}{\Gamma(\rho)n!}R_N^{-\alpha n-\beta}(x) \\ =\frac{(-1)^N(x-c)^{N+\beta+1}}{N!}E_{\alpha,\beta}^{\rho,\kappa}\left(\omega(x-c)^{\alpha}\right)\int_0^1\int_0^1f\big(c+\eta(x-c)\big)g^{(N+1)}\big(c+(\zeta+\eta-\zeta\eta)(x-c)\big)\,\mathrm{d}\eta\,\mathrm{d}\zeta,
\end{multline}
where the integrand is independent of $n$. We can then ignore the generalised Mittag-Leffler function when taking the limit, because it is independent of $N$, and we find that \eqref{product:limit} holds just as in \cite{podlubny,baleanu-fernandez}.
\end{proof}

\begin{example}
As a simple application of the Prabhakar product rule, we apply the result of Theorem \ref{product:GP} with $f(x)=e^{ax}$ and $g(x)=x$ and the constant of differintegration $c=i\infty$.

In this case, the outer series in \eqref{product:GP:eqn} has only two non-trivial terms, namely $m=0$ and $m=1$, while the Riemann--Liouville fractional integral of an exponential function is well known. Thus the series becomes:
\begin{align*}
&{\color{white}=}\sum_{m=0}^{1}\frac{\mathrm{d}^mx}{\mathrm{d}x^m}\left[\sum_{n=0}^{\infty}\frac{\Gamma(\rho+\kappa n)\Gamma(1-\beta-\alpha n)\omega^n}{\Gamma(\rho)\Gamma(1-\beta-\alpha n-m)m!n!}a^{\alpha n+\beta+m}e^{ax}\right] \\
\begin{split}
&=x\left[\sum_{n=0}^{\infty}\frac{\Gamma(\rho+\kappa n)\Gamma(1-\beta-\alpha n)(\omega a^{\alpha})^n}{\Gamma(\rho)\Gamma(1-\beta-\alpha n)n!}a^{\beta}e^{ax}\right] \\
&\hspace{5cm}+\left[\sum_{n=0}^{\infty}\frac{\Gamma(\rho+\kappa n)\Gamma(1-\beta-\alpha n)(\omega a^{\alpha})^n}{\Gamma(\rho)\Gamma(-\beta-\alpha n)n!}a^{\beta+1}e^{ax}\right]
\end{split} \\
&=x\left[\sum_{n=0}^{\infty}\frac{\Gamma(\rho+\kappa n)(\omega a^{\alpha})^n}{\Gamma(\rho)n!}a^{\beta}e^{ax}\right]+\left[\sum_{n=0}^{\infty}\frac{\Gamma(\rho+\kappa n)(-\beta-\alpha n)(\omega a^{\alpha})^n}{\Gamma(\rho)n!}a^{\beta+1}e^{ax}\right] \\
&=\sum_{n=0}^{\infty}\frac{\Gamma(\rho+\kappa n)(\omega a^{\alpha})^n}{\Gamma(\rho)n!}(x-a\beta-a\alpha n)a^{\beta}e^{ax},
\end{align*}
and we have computed the Prabhakar fractional integral $\mathcal{E}_{\alpha,\beta}^{\omega,\rho,\kappa}\left(xe^x\right)$.
\end{example}

Naturally, Theorem \ref{product:GP} yields corollaries in the form of product rule analogues for other more specific fractional models, including \eqref{product:AB} for the AB model, by using Proposition \ref{GPcases}.

\section{Fractional extensions of the chain rule}

The chain rule is another fundamental result from classical calculus which can be generalised to certain models of fractional calculus. Another paper of Osler \cite{osler2} considered various generalisations of the chain rule to the Riemann--Liouville model; the following version follows directly from the fractional product rule \eqref{product:RL} and the classical Fa\`a di Bruno formula:
\begin{multline}
\label{chain:RL}
\prescript{RL}{}D^{\alpha}_{c+}f(g(x))=\frac{(x-c)^{-\alpha}}{\Gamma(1-\alpha)}f(g(x)) \\ +\sum_{m=1}^{\infty}\binom{\alpha}{m}\frac{(x-c)^{m-\alpha}}{\Gamma(m-\alpha+1)}\sum_{r=1}^m\frac{\mathrm{d}^rf(g(x))}{\mathrm{d}g(x)^r}\sum_{(P_1,\dots,P_m)}\Bigg[\prod_{j=1}^m\tfrac{j}{P_j!(j!)^{P_j}}\Big(\tfrac{\mathrm{d}^jg(x)}{\mathrm{d}x^j}\Big)^{P_j}\Bigg],\quad x,\alpha,c\in\mathbb{C},
\end{multline}
where $g$ is smooth and $f(g(x))$ is a function of the form $x^{\zeta}\xi(x)$ with $\mathrm{Re}(\zeta)>-1$ and $\xi$ holomorphic on a complex domain $U\ni c$, and the final summation is over the set
\begin{equation}
\label{chain:summation}
\{(P_1,\dots,P_m)\in\left(\mathbb{Z}^+_0\right)^m:\sum_jP_j=r,\sum_jjP_j=m\Big\}.
\end{equation}
Using the series formula for AB fractional derivatives, it is once again possible \cite{baleanu-fernandez} to derive from \eqref{chain:RL} a version of the fractional chain rule which is valid in the AB model. Here we shall use our new series formula to derive a similar result in the generalised Prabhakar model of fractional calculus.

\begin{theorem}
\label{chain:GP}
Let $f$ and $g$ be complex functions such that $g$ is smooth and $f(g(x))$ is a function of the form $x^{\zeta}\xi(x)$ with $\mathrm{Re}(\zeta)>-1$ and $\xi$ holomorphic on a complex domain $U\subset\mathbb{C}$. Then for any complex parameters $\alpha,\beta,\omega,\rho,\kappa$ satisfying the conditions \eqref{GGML:conditions}, the generalised Prabhakar operator satisfies the following version of the chain rule:
\begin{multline}
\label{chain:GP:eqn}
\mathcal{E}_{\alpha,\beta;c+}^{\omega,\rho,\kappa}\big(f(g(x))\big)=(x-c)^{\beta}\sum_{n=0}^{\infty}\frac{\Gamma(\rho+\kappa n)\Gamma(1-\beta-\alpha n)\left[\omega(x-c)^{\alpha}\right]^n}{\Gamma(\rho)\Gamma(\alpha n+\beta)n!}\sum_{m=0}^{\infty} \\
\frac{(c-x)^m}{m!(\alpha n+\beta+m)}\left[\sum_{r=1}^m\frac{\mathrm{d}^rf(g(x))}{\mathrm{d}g(x)^r}\sum_{(P_1,\dots,P_m)}\Bigg[\prod_{j=1}^m\tfrac{j}{P_j!(j!)^{P_j}}\Big(\tfrac{\mathrm{d}^jg(x)}{\mathrm{d}x^j}\Big)^{P_j}\Bigg]\right],
\end{multline}
where the summation over $(P_1,\dots,P_m)$ is over the set \eqref{chain:summation}.
\end{theorem}

\begin{proof}
Our starting point is the result of Theorem \ref{product:GP}, which in this case we apply to the product of the two functions $f(g(x))$ and $I(x)=1$. The fractional differintegrals of the unit function $I$ are well known, so the expression \eqref{product:GP:eqn} becomes:
\begin{align*}
\mathcal{E}_{\alpha,\beta;c+}^{\omega,\rho,\kappa}\big(f(g(x))\big)&=\sum_{m=0}^{\infty}\frac{\mathrm{d}^m(f\circ g)}{\mathrm{d}x^m}\left[\sum_{n=0}^{\infty}\frac{\Gamma(\rho+\kappa n)\Gamma(1-\beta-\alpha n)\omega^n}{\Gamma(\rho)\Gamma(1-\beta-\alpha n-m)m!n!}\prescript{RL}{}I_{c+}^{\alpha n+\beta+m}(1)\right] \\
&=\sum_{m=0}^{\infty}\frac{\mathrm{d}^m(f\circ g)}{\mathrm{d}x^m}\left[\sum_{n=0}^{\infty}\frac{\Gamma(\rho+\kappa n)\Gamma(1-\beta-\alpha n)\omega^n}{\Gamma(\rho)\Gamma(1-\beta-\alpha n-m)m!n!}\cdot\frac{(x-c)^{\alpha n+\beta+m}}{\Gamma(\alpha n+\beta+m+1)}\right] \\
&=\sum_{m=0}^{\infty}\frac{\mathrm{d}^m(f\circ g)}{\mathrm{d}x^m}\left[\sum_{n=0}^{\infty}\frac{\Gamma(\rho+\kappa n)\Gamma(1-\beta-\alpha n)\left[\omega(x-c)^{\alpha}\right]^n(x-c)^{\beta+m}}{\Gamma(\rho)\frac{\pi(\alpha n+\beta+m)}{\sin\left(\pi(\alpha n+\beta+m)\right)}m!n!}\right] \\
&=(x-c)^{\beta}\sum_{m=0}^{\infty}\frac{\mathrm{d}^m(f\circ g)}{\mathrm{d}x^m}\left[\sum_{n=0}^{\infty}\frac{\Gamma(\rho+\kappa n)\Gamma(1-\beta-\alpha n)\left[\omega(x-c)^{\alpha}\right]^n}{\Gamma(\rho)\frac{\pi(\alpha n+\beta+m)}{\sin\left(\pi(\alpha n+\beta)\right)}n!}\cdot\frac{(c-x)^m}{m!}\right] \\
&=(x-c)^{\beta}\sum_{m=0}^{\infty}\frac{\mathrm{d}^m(f\circ g)}{\mathrm{d}x^m}\left[\sum_{n=0}^{\infty}\frac{\Gamma(\rho+\kappa n)\Gamma(1-\beta-\alpha n)\left[\omega(x-c)^{\alpha}\right]^n}{(\alpha n+\beta+m)\Gamma(\rho)\Gamma(\alpha n+\beta)n!}\cdot\frac{(c-x)^m}{m!}\right] \\
&=(x-c)^{\beta}\sum_{n=0}^{\infty}\frac{\Gamma(\rho+\kappa n)\Gamma(1-\beta-\alpha n)\left[\omega(x-c)^{\alpha}\right]^n}{\Gamma(\rho)\Gamma(\alpha n+\beta)n!}\sum_{m=0}^{\infty}\frac{\mathrm{d}^m(f\circ g)}{\mathrm{d}x^m}\cdot\frac{(c-x)^m}{m!(\alpha n+\beta+m)} \\
\begin{split}
&=(x-c)^{\beta}\sum_{n=0}^{\infty}\frac{\Gamma(\rho+\kappa n)\Gamma(1-\beta-\alpha n)\left[\omega(x-c)^{\alpha}\right]^n}{\Gamma(\rho)\Gamma(\alpha n+\beta)n!}\sum_{m=0}^{\infty} \\
&\hspace{2cm}\frac{(c-x)^m}{m!(\alpha n+\beta+m)}\left[\sum_{r=1}^m\frac{\mathrm{d}^rf(g(x))}{\mathrm{d}g(x)^r}\sum_{(P_1,\dots,P_m)}\Bigg[\prod_{j=1}^m\tfrac{j}{P_j!(j!)^{P_j}}\Big(\tfrac{\mathrm{d}^jg(x)}{\mathrm{d}x^j}\Big)^{P_j}\Bigg]\right],
\end{split}
\end{align*}
as required, where we used the classical Fa\`a di Bruno formula in the final step.

\end{proof}

Once again, we can use Theorem \ref{chain:GP} together with Proposition \ref{GPcases} in order to find analogues of the chain rule for other more specific fractional models, including the result already established in \cite{baleanu-fernandez} for the AB model.

We note that one application of Theorem \ref{chain:GP} would be to compute fractional differintegrals, in the various models covered by the umbrella of Prabhakar, of a Gaussian function $e^{-x^2}$. This follows from putting $f(x)=e^x$ and $g(x)=-x^2$ in the identity \eqref{chain:GP:eqn}.

\section{Iteration of fractional models}

In \cite{fernandez-baleanu}, the idea of \textit{iteration} was used to propose a new model of fractional calculus. Given a functional operator, we can consider iterating it a natural number of times, and then try to find a way of defining the $\nu$th iteration for $\nu$ not an integer. This is the basic idea underlying fractional calculus, and the same idea can be applied even to operators which are themselves interpreted as fractional differintegrals. Here we shall explore how the same methodology can be applied to operators in the Prabhakar model.

By Theorem \ref{PRABsemigp}, the iterated Prabhakar operator looks like the following:
\begin{equation*}
\left(\mathcal{E}_{\alpha,\beta;c+}^{\omega,\rho}\right)^nf=\mathcal{E}_{\alpha,n\beta;c+}^{\omega,n\rho}f,
\end{equation*}
for any $n\in\mathbb{N}$ and any $L^1$ function $f$. Thus, the natural definition of the fractionally iterated Prabhakar operator is:
\begin{equation}
\label{IPdef:int}
\left(\mathcal{E}_{\alpha,\beta;c+}^{\omega,\rho}\right)^{\nu}f=\mathcal{E}_{\alpha,\nu\beta;c+}^{\omega,\nu\rho}f,
\end{equation}
for $\nu\in\mathbb{C}$ such that $\mathrm{Re}(\nu\beta)>0$. If $\nu$ is such that the latter condition does not hold, then we can define the $\nu$th iteration of the Prabhakar integral in the same way as we defined the Prabhakar derivative \eqref{PRABdef:deriv}, which corresponds to the case $\nu=-1$. In other words, we create the following definition.

\begin{definition}
Under the usual conditions $\mathrm{Re}(\alpha)>0,\mathrm{Re}(\beta)>0$, and for any $\nu\in\mathbb{C}$, the $\nu$th iteration of the Prabhakar operator is defined as:
\begin{equation}
\label{IPdef}
\left(\mathcal{E}_{\alpha,\beta;c+}^{\omega,\rho}\right)^{\nu}f(x)\coloneqq\frac{\mathrm{d}^m}{\mathrm{d}x^m}\mathcal{E}_{\alpha,m+\nu\beta;c+}^{\omega,\nu\rho}f(x),\quad m\coloneqq\max(0,\lfloor\mathrm{Re}(-\nu\beta)\rfloor+1).
\end{equation}
We see that this equates with the formula \eqref{IPdef:int} when $\mathrm{Re}(\nu\beta)>0$. It also equates with the formula \eqref{PRABdef:deriv} for Prabhakar fractional derivatives when $\nu=-1$.
\end{definition}

Using the series formula from section 2, we can present a unified expression for the iterated Prabhakar differintegral, which applies equally well regardless of the sign of $\mathrm{Re}(\nu\beta)$.

\begin{theorem}
\label{IPseries}
For $\alpha,\beta,\omega,\rho,\nu\in\mathbb{C}$ with $\mathrm{Re}(\alpha)>0,\mathrm{Re}(\beta)>0$, the iterated Prabhakar differintegral defined by \eqref{IPdef} can be defined as
\begin{equation}
\label{IPseries:eqn}
\left(\mathcal{E}_{\alpha,\beta;c+}^{\omega,\rho}\right)^{\nu}f(x)=\sum_{n=0}^{\infty}\frac{\Gamma(\nu\rho+n)\omega^n}{\Gamma(\nu\rho)n!}\prescript{RL}{}I_{c+}^{\alpha n+\nu\beta}f(x),
\end{equation}
where the series on the right-hand side is locally uniformly convergent.
\end{theorem}

\begin{proof}
For $\mathrm{Re}(\nu\beta)>0$, the result follows directly from substituting the series \eqref{PRABseries:int} into the identity \eqref{IPdef:int} which is the integral incarnation of the definition \eqref{IPdef}.

For $\mathrm{Re}(\nu\beta)\leq0$, the parameter $n$ defined in \eqref{IPdef} is strictly positive, and so substituting the series \eqref{PRABseries:int} into \eqref{IPdef} yields:
\begin{align*}
\left(\mathcal{E}_{\alpha,\beta;c+}^{\omega,\rho}\right)^{\nu}f(x)&=\frac{\mathrm{d}^m}{\mathrm{d}x^m}\mathcal{E}_{\alpha,m+\nu\beta;c+}^{\omega,\nu\rho}f(x) \\
&=\frac{\mathrm{d}^m}{\mathrm{d}x^m}\left(\sum_{n=0}^{\infty}\frac{\Gamma(\nu\rho+n)\omega^n}{\Gamma(\nu\rho)n!}\prescript{RL}{}I_{c+}^{\alpha n+m+\nu\beta}f(x)\right) \\
&=\sum_{n=0}^{\infty}\frac{\Gamma(\nu\rho+n)\omega^n}{\Gamma(\nu\rho)n!}\prescript{RL}{}I_{c+}^{\alpha n+\nu\beta}f(x),
\end{align*}
as required.
\end{proof}

We emphasise that using a series formula enables the definitions of fractional derivatives and integrals to be merged into a single unified expression for fractional differintegrals. The same phenomenon has been observed with Grunwald--Letnikov differintegrals, and also with the standard Prabhakar operators in our remark after Corollary \ref{PRABseries}. This is one reason why convergent series expressions have been so important in the study of fractional calculus.

Another important property of the iterated Prabhakar differintegral is that it satisfies a semigroup property in $\nu$, as shown by the following result.

\begin{theorem}[Semigroup property]
\label{IPsemigp}
Given parameters $\alpha,\beta,\omega,\rho\in\mathbb{C}$ with $\mathrm{Re}(\alpha)>0,\mathrm{Re}(\beta)>0$, we have the following identity valid for any $L^1$ function $f$:
\begin{equation}
\label{IPsemigp:eqn}
\left(\mathcal{E}_{\alpha,\beta;c+}^{\omega,\rho}\right)^{\mu}\left(\mathcal{E}_{\alpha,\beta;c+}^{\omega,\rho}\right)^{\nu}f(x)=\left(\mathcal{E}_{\alpha,\beta;c+}^{\omega,\rho}\right)^{\mu+\nu}f(x),\quad\mu,\nu\in\mathbb{C},\mathrm{Re}(\nu\beta)>0.
\end{equation}
\end{theorem}

\begin{proof}
This follows from the formula \eqref{IPdef:int} for iterated Prabhakar integrals and the result of Theorem \ref{PRABsemigp}.
\end{proof}

\section{Conclusions}

We proved in this paper that the Prabhakar fractional model and its generalised forms can be reduced to series involving only Riemann--Liouville integrals. This basic result is due to the fact that the Mittag-Leffler function and its generalised versions have nice series representations which are locally uniformly convergent.

These results are significant for fractional calculus from a philosophical point of view: they demonstrate that several alternative models of fractional differintegrals can be written in terms of only the classical model. They are also practically significant, since they enable or streamline the proofs of many fundamental results about the Prabhakar and other models of fractional calculus: not only previously-known facts such as semigroup and inverse properties, but also new results such as product and chain rules.

One reason for the power of the Prabhakar model is that it satisfies a semigroup property in some of the parameters. We have demonstrated how this fact can be linked to the idea of fractional iteration, which is an important concept to consider in fractional calculus.

\section*{Acknowledgements}

The authors would like to thank both the editor and the anonymous referees for their helpful comments and suggestions.

\end{document}